\documentclass[a4paper,11pt,reqno]{amsart}

\usepackage{graphicx}
\usepackage{a4wide}
\usepackage{eucal}
\usepackage[pdftex,colorlinks]{hyperref}

\numberwithin{equation}{section}

\newcommand{\R}{\mathbb R}
\newcommand{\C}{\mathbb C}
\newcommand{\N}{\mathbb N}
\newcommand{\Z}{\mathbb Z}

\newcommand{\be}{\begin{equation}}
\newcommand{\ee}{\end{equation}}
\newcommand{\ba}{\begin{eqnarray}}
\newcommand{\ea}{\end{eqnarray}}

\newtheorem{theorem}{Theorem}[section]
\newtheorem{proposition}[theorem]{Proposition}
\newtheorem{remark}[theorem]{Remark}
\newtheorem{lemma}[theorem]{Lemma}
\newtheorem{corollary}[theorem]{Corollary}

\begin{document}

\title[Null controllability of one-dimensional parabolic equations]{Null controllability of one-dimensional parabolic equations by  the flatness approach}

\author{Philippe Martin}
\address{Centre Automatique et Syst\`emes, MINES ParisTech, PSL Research University\\
60 boulevard Saint-Michel, 75272 Paris Cedex 06, France}
\email{philippe.martin@mines-paristech.fr}

\author{Lionel Rosier}
\address{Centre Automatique et Syst\`emes, MINES ParisTech, PSL Research University\\
60 boulevard Saint-Michel, 75272 Paris Cedex 06, France}
\email{lionel.rosier@mines-paristech.fr}

\author{Pierre Rouchon}
\address{Centre Automatique et Syst\`emes, MINES ParisTech, PSL Research University\\
60 boulevard Saint-Michel, 75272 Paris Cedex 06, France}
\email{pierre.rouchon@mines-paristech.fr}

\subjclass{}
\begin{abstract}                          
We consider linear one-dimensional parabolic equations with space dependent coefficients that are only measurable and that may be degenerate or singular.
Considering generalized Robin-Neumann boundary conditions at both extremities, we prove the null controllability with one boundary control by following the flatness approach, which provides
explicitly the control and the associated trajectory as series. Both the control and the trajectory have a Gevrey regularity in time related to the $L^p$ class of the coefficient in front of $u_t$.
The approach applies in particular to the (possibly degenerate or singular) heat equation $(a(x)u_x)_x-u_t=0$ with $a(x)>0$ for a.e.  $x\in (0,1)$ and $a+1/a \in L^1(0,1)$, or to the heat equation with inverse square 
potential $u_{xx}+(\mu  / |x|^2)u-u_t=0$
with $\mu\ge 1/4$. 
\end{abstract}

\keywords{degenerate parabolic equation; singular coefficient; null controllability, Gevrey functions; flatness}

\maketitle
\section{Introduction}
The null controllability of parabolic equations has been extensively investigated since several decades. After the pioneering work in \cite{FR,jones, LK}, mainly concerned with the 
one-dimensional case, there has been significant  progress in the general N-dimensional case \cite{FI,imanuvilov,LR} by using Carleman estimates.
The more recent developments of the theory were concerned with discontinuous coefficients \cite{AE,BDL,FZ,lerousseau}, degenerate coefficients \cite{ACF,BCG,BL,CFR,CMV04,CMV08,CMV09,FT}, or singular coefficients \cite{cazacu,ervedoza,VZ}.   

In \cite{AE}, the authors derived the null controllability of a linear one-dimensional parabolic equation with (essentially bounded) measurable coefficients. The method of proof combined the
Lebeau-Robbiano  approach \cite{LR} with some complex analytic arguments. 
 
Here, we are concerned with the null controllability of the system 
\ba
(a(x)u_x)_x +b(x)u_x + c(x)u-\rho (x) u_t&=&0, \quad x\in (0,1),\ t\in (0,T), \label{B1}\\
\alpha _0u(0,t)+\beta _0 (au_x)(0,t)&=& 0, \quad t\in (0,T), \label{B2}\\
\alpha _1u(1,t)+\beta _1 (au_x)(1,t)&=& h(t), \quad t\in (0,T), \label{B3}\\
u(x,0)&=& u_0(x), \quad x\in (0,1), \label{B4}
\ea
where $(\alpha _0,\beta _0),(\alpha _1,\beta _1)\in \R ^2 \setminus \{ (0,0)\} $ are given, 
$u_0\in L^2(0,1)$ is the initial state and $h\in L^2(0,T)$ is the control input. 

The given functions $a,b,c,\rho $ will be assumed to fulfill the following conditions 
\ba
&&a(x) >0  \textrm{ and } \rho (x)>0 \textrm{ for a.e. } x\in (0,1), \label{B11}\\
&&(\frac{1}{a},\frac{b}{a}, c, \rho ) \in[ L^1(0,1) ]^4,\label{B12}\\
&&\exists K\ge 0, \quad \frac{c(x)}{\rho (x)}\leq K \  \textrm{ for a.e. } x\in (0,1),\label{B13}\\
&&\exists p\in (1,\infty ],\quad 
a^{1-\frac{1}{p} } \rho \in L^p(0,1). \label{B14}
\ea
The assumptions \eqref{B11}-\eqref{B14} are clearly less restrictive than the assumptions from \cite{AE}: 
\be
a,b,c,\rho\in L^\infty(0,1) \textrm{ and } a(x)>\varepsilon, \ \rho (x) > \varepsilon>0 \textrm{ for a.e. } x\in (0,1) 
\ee
for some $\varepsilon >0$.

Let us introduce some notations. Let $B$ be a Banach space with norm $\Vert \cdot \Vert _B$. For any $t_1<t_2$ and $s\ge 0$, we denote 
by  $G^s([t_1,t_2],B)$ the class of (Gevrey) functions $u\in C^\infty ([t_1,t_2],B)$ for which there exist positive constants $M,R$ such that 
  \be
  \label{BB1}
  \Vert u^{(p)}(t)\Vert _B \le M\frac{p!^s}{R^p}\quad \forall t\in [t_1,t_2], \ \forall p\ge 0.
  \ee 
  When $(B,\Vert \cdot\Vert _B)=(\R , |\cdot |)$, $G^s([t_1,t_2],B)$ is merely denoted $G^s([t_1,t_2] )$. 
  Let 
 \[
 L^1_\rho :=\{ u:(0,1)\to \R; \ \ || u ||_{L^1_\rho} :=\int_0^1 |u(x)| \rho (x) dx <\infty \} .
 \]
Note that $L^2(0,1)\subset L^1_\rho$ if $\rho\in L^2(0,1)$.
The main result in this paper is the following 
\begin{theorem}
\label{thm1}
Let the functions $a,b,c,\rho:(0,1)\to \R$ satisfy \eqref{B11}-\eqref{B14} for some numbers $K\ge 0$, $p \in (1,\infty ]$. 
Let $(\alpha _0,\beta _0),(\alpha _1,\beta _1)\in \R ^2 \setminus \{ (0,0)\} $ and $T>0$. 
Pick any $u_0\in L^1_\rho$ and any $s\in (1,2-1/p)$. Then there exists a function $h\in G^s([0,T])$, that may be given explicitly as a series, such that 
the solution $u$ of \eqref{B1}-\eqref{B4} satisfies $u(.,T)=0$. Moreover $u\in G^s([\varepsilon ,T], W^{1,1}(0,1))$ and 
$au_x\in G^s([\varepsilon ,T], C^0([0,1]))$  for all $\varepsilon \in (0,T)$. 
\end{theorem}
Clearly, Theorem \ref{thm1} can be applied to parabolic equations with discontinuous coefficients that may be degenerate or singular at a point (or more generally at a sequence of points). 
The proof of it is not based on the classical duality
approach, in the sense that it does not rely on the proof of some observability inequality for the adjoint equation. It follows the flatness approach developed 
in \cite{Laroc2000PhD,LarocM2000MTNS,LMR,MRR1D,MRR2D,
MRRND,Meurer2012book}. This direct approach gives explicitly both the control and the trajectory as series, which leads to efficient numerical schemes by taking partial sums in the series \cite{MRRND, MRRPE}. 
Let us describe its main steps. 
In the first step, following \cite{AE}, we show that after a series of changes of dependent/independent variables, system \eqref{B1}-\eqref{B4} can be put into the canonical form 
\ba
u_{xx} -\rho (x) u_t&=&0, \quad x\in (0,1),\ t\in (0,T), \label{C1}\\
\alpha _0u(0,t)+\beta _0 u_x(0,t)&=& 0, \quad t\in (0,T), \label{C2}\\
\alpha _1u(1,t)+\beta _1 u_x(1,t)&=& h(t), \quad t\in (0,T), \label{C3}\\
u(x,0)&=& u_0(x), \quad x\in (0,1), \label{C4}
\ea
where $\rho (x)>0$ a.e. in $(0,1)$ and $\rho \in L^p(0,1)$ with $p\in (1,\infty ]$. 
In the second step, following \cite{MRR1D, MRRND}, we seek $u$ in the form 
\ba
\label{C5} u(x,t)&=& \sum_{n\ge 0} e^{-\lambda _n t} e_n(x), \quad x\in (0,1), \ t\in [0,\tau ], \\ 
\label{C6} u(x,t) &=& \sum_{i\ge 0} y^{ (i) } (t) g _i (x) , \quad  x\in (0,1), \ t\in [\tau ,T], 
\ea 
where $\tau\in (0,T)$ is any intermediate time; $e_n:(0,1)\to \R$  (resp. $\lambda _n\in \R$) denotes the $n^{th}$ {\em eigenfunction} (resp. {\em eigenvalue}) 
associated with \eqref{C1}-\eqref{C3} and satisfying  \cite{Laroc2000PhD,LarocM2000MTNS}
\ba
\label{C7} -e_n'' &=& \lambda _n\,  \rho \, e_n, \quad x\in (0,1)\\
\label{C8} \alpha _0e_n(0)+\beta _0e_n'(0)&=&0,\\
\label{C9} \alpha _1e_n(1)+\beta _1e_n'(1)&=&0,
\ea
while $g_i:(0,1)\to\R$ is defined inductively as the solution to the Cauchy problem
\ba
\label{C10} g_0'' &=& 0, \quad x\in (0,1)\\
\label{C11} \alpha _0 g_0 (0)+\beta _0  g_0 '(0)&=&0,\\
\label{C12} \beta _0 g_0 (0) - \alpha _0 g_0 '(0)&=&1
\ea
for $i=0$, and to the Cauchy problem
\ba
\label{C13} g_i'' &=& \rho\,  g_{i-1}, \quad x\in (0,1)\\
\label{C14} g_i(0)&=&0,\\
\label{C15} g_i'(0)&=&0
\ea
for $i\ge 1$. Expanding $u$ on generating functions as in~\eqref{C6} rather than on powers of~$x$ as in~\cite{LMR,MRR1D} was 
introduced in~\cite{LarocM2000MTNS} and studied in~\cite{Laroc2000PhD}.

The fact that the generating function $g_i$ is defined as the solution of a {\em Cauchy problem}, rather than the solution of a {\em boundary-value problem}, is crucial in the analysis 
developed here. First, it allows to prove that {\em every} initial state in the space $L^1_\rho$ (and not only states in some restricted class of Gevrey functions)
can be driven to 0 in time $T$. Secondly, from \eqref{C13}-\eqref{C15},  we see by an easy induction on $i$ that for $\rho \in L^\infty(0,1)$, the function $g_i$ is uniformly bounded
by $ C/(2i)!$, and hence the series in \eqref{C6} is indeed convergent when $y\in G^s([\tau ,T])$ with $1<s<2$.
 
The corresponding control function $h$ is given explicitly as  
\[h(t) =
\left\{  
\begin{array}{ll} 
0 \quad &\textrm{ if } 0 \le t\le \tau,\\
\sum_{i\ge 0} y^{ (i) } (t) ( \alpha _1 g_i (1) + \beta _1 g_i ' (1) )  \quad &\textrm{ if } \tau <  t\le T.  
\end{array}
\right.
\]

It is easy to see that the function $u(x,t)$ defined in \eqref{C6} satisfies (formally) \eqref{C1}, and also the condition $u(x,T)=0$
if $y^{(i)}(T)=0$ for all $i\in \N$, so that the null controllability can be established for {\em some} initial states. 
The main issue is then to extend it to {\em every} initial state $u_0\in L^1_\rho$. Following \cite{MRR1D, MRR2D,MRRND}, 
we first use the strong smoothing effect of the heat equation to smooth out the state function in the time interval $(0,\tau)$. Next, to ensure that the two expressions of $u$ given
in \eqref{C5}-\eqref{C6} coincide at $t=\tau$, we have to relate the eigenfunctions $e_n$ to the generating functions $g_i$. 

It will be shown that any eigenfunction $e_n$ can be expanded in terms of the generating functions $g_i$ as
\be
\label{C16} 
e_n(x)= \zeta_n \sum_{i\ge 0} (-\lambda _n)^i g_i (x)
\ee 
with $\zeta _n\in \R$. Note that, for $\rho \equiv 1$ and $(\alpha _0,\beta _0,\alpha _1,\beta _1)=(0,1,0,1)$,  
$\lambda _n =(n\pi )^2$ for all $n\ge 0$, $e_0(x)=1$ and $e_n(x)=\sqrt{2} \cos (n\pi x)$ for $n\ge 0$ while
$g_i(x) = x^{2i}/(2i)!$, so that \eqref{C16} for $n\ge 1$ is nothing but the classical Taylor expansion of $\cos (n\pi x)$ around $x=0$:
\be
\label{C16bis}
\cos (n\pi x ) = \sum_{i\ge 0} (-1)^i \frac{(n\pi x) ^{2i} }{(2i)!}\cdot
\ee
Thus \eqref{C16} can be seen as a natural extension of \eqref{C16bis}, in which the generating 
functions $g_i$, {\em a priori} not smoother than $W^{2,p}(0,1)$, replace the functions $x^{2i}/(2i)!$. 
 
The condition \eqref{B14} is used to prove the estimate 
\[
|g_i(x)| \le \frac{C}{R^{2i} (i!)^{2-\frac{1}{p}} }
\]
needed to ensure the convergence of the series in \eqref{C6} when $y\in G^s( [\tau ,T])$ with $1<s<2-1/p$. 

Theorem \ref{thm1} applies in particular to any system 
\ba
(a(x)u_x)_x -u_t&=&0, \quad x\in (0,1),\ t\in (0,T), \label{B31}\\
\alpha _0u(0,t)+\beta _0 (au_x)(0,t)&=& 0, \quad t\in (0,T), \label{B32}\\
\alpha _1u(1,t)+\beta _1 (au_x)(1,t)&=& h(t), \quad t\in (0,T), \label{B33}\\
u(x,0)&=& u_0(x), \quad x\in (0,1), \label{B34}
\ea
where $a(x)>0$ for a.e. $x\in (0,1)$ and $a+1/a\in L^1(0,1)$. (Pick $p=2$ in \eqref{B14}.) This includes the case where $a$ is measurable, positive and essentially bounded together with its inverse 
(but not necessarily piecewise continuous),
and the case where $a(x)=x^r$ with $-1<r<1$. (Actually any $r\le -1$ is also admissible, by picking $p>1$ sufficiently close to $1$ in \eqref{B14}.) Note that our result applies as well to  $a(x)=(1-x)^r$ with $0<r<1$, yielding a positive
null controllability result when the control is applied at the point ($x=1$) where the diffusion coefficient degenerates (see \cite[Section 2.7]{CMV08}). 
Note also that the coefficient $a(x)$ is allowed to be degenerate/singular at a {\em sequence} of points: consider e.g. $a(x):=|\sin (x^{-1})|^r$ with $-1<r<1$. Then $a+1/a\in L^1(0,1)$.  

The null controllability of \eqref{B31}-\eqref{B34} for $a(x)=x^r$ with $0<r<2$ was established (in appropriate spaces)
in \cite{CMV08}. The situation when $1/a\not\in L^1(0,1)$ (e.g. $a(x)=x^r$ with $1\le r<2$) is beyond these notes, and it will be considered elsewhere.

A null controllability result with an internal control can be deduced from \eqref{thm1}. Its proof is given in appendix, for the sake of completeness.
\begin{corollary}
\label{cor-intro}
Assume given an open set $\omega =(l_1,l_2)$ with $0<l_1<l_2<1$, and let us consider the following
control system
  \ba
(a(x)u_x)_x +b(x)u_x + c(x)u-\rho (x) u_t&=&\chi _{\omega }f(x,t), \quad x\in (0,1),\ t\in (0,T), \label{B100}\\
\alpha _0u(0,t)+\beta _0 (au_x)(0,t)&=& 0, \quad t\in (0,T), \label{B200}\\
\alpha _1u(1,t)+\beta _1 (au_x)(1,t)&=& 0, \quad t\in (0,T), \label{B300}\\
u(x,0)&=& u_0(x), \quad x\in (0,1), \label{B400}
\ea
where $u_0\in L^1_\rho$ is any given initial data,  and $a$, $b$, $c$, $\rho$, $p$, $K$, $s$, $(\alpha _0,\beta _0)$ and $(\alpha _1, \beta _1)$ are as in 
Theorem \ref{thm1}.  Then one can find a control input $f\in L^2(0,T,L^2_a(\omega ))$ such that the solution $u$ of \eqref{B100}-\eqref{B400} satisfies $u(x,T)=0$ for all $x\in (0,1)$.  
\end{corollary}

Another important family of heat equations with variable coefficients is those with inverse square potential localized at the boundary, namely
\ba
u_{xx} +\frac{\mu}{x^2}u -u_t&=&0, \quad x\in (0,1),\ t\in (0,T), \label{B41}\\
u(0,t)&=& 0, \quad t\in (0,T), \label{B42}\\
\alpha _1u(1,t)+\beta _1 u_x(1,t)&=& h(t), \quad t\in (0,T), \label{B43}\\
u(x,0)&=& u_0(x), \quad x\in (0,1), \label{B44}
\ea
where $\mu\in \R$ is a given number. Note that Theorem \ref{thm1} cannot be applied to \eqref{B41}-\eqref{B44}, 
for $c(x)=\mu x^{-2}$ is not integrable on $(0,1)$. It was proved in \cite{cazacu} that \eqref{B41}-\eqref{B44} is null controllable
in $L^2(0,1)$ when $\mu \le 1/4$ by combining Carleman inequalities to Hardy inequalities. We shall show in this paper that this result
can be retrieved by the flatness approach as well. 

\begin{theorem}
\label{thm2} 
Let $\mu \in (0,1/4]$, $(\alpha _1, \beta _1)\in \R ^2 \setminus \{ (0,0) \} $, $T>0$, and $\tau \in (0,T)$. Pick any $u_0\in L^2(0,1)$ and any $s\in (1,2)$. Then there 
exists a function $h\in G^s([0,T])$  with $h(t)=0$ for $0\le  t \le \tau$ and such that the solution $u$ of \eqref{B41}-\eqref{B44} satisfies $u(T,.)=0$. Moreover,
$u\in G^s([\varepsilon , T], W^{1,1}(0,1))$  for all $\varepsilon \in (0,T)$. Finally, if $0\le \mu <1/4$ and $r>(1+\sqrt{1- 4 \mu} )/2$, 
then $x^ru_x\in  G^s([\varepsilon , T], C^0([0,1]))$  for all $\varepsilon \in (0,T)$. 
\end{theorem}
  
The paper is organized as follows. Section 2 is devoted to the proof of Theorem \ref{thm1}. We first show that a convenient change of variables 
transforms \eqref{B1}-\eqref{B4} into \eqref{C1}-\eqref{C4} (Proposition \ref{prop1}). Next, we show that the flatness approach can be applied to \eqref{C1}-\eqref{C4}
to yield a null controllability result (Theorem \ref{thm3}). Performing the inverse change of variables, we complete the proof of Theorem  \ref{thm1}.  Section 3 contains the proof of Theorem \ref{thm2}, which is obtained
as a consequence of Theorem \ref{thm3} after a convenient change of variables,  and 
some examples.  

\section{Proof of Theorem \ref{thm1}}
\subsection{Reduction to the canonical form \eqref{C1}-\eqref{C4}}
\label{section21}
Let $a,b,c,\rho$, and $p$ be as in \eqref{B11}-\eqref{B14}. Set 
\begin{eqnarray}
B(x)&:=& \int_0^x \frac{b(s)}{a(s)} ds,  \label{XXX1}\\
\tilde a(x) &:=& a(x) e^{B(x)} \label{XXX2}\\
\tilde c(x) &:=& (K\rho (x) -c(x))e^{B(x)}. \label{XXX3}
\end{eqnarray}
Then $B\in W^{1,1}(0,1)$, $\tilde c\in L^1(0,1)$, and 
\[
\tilde a(x)>0 \textrm{ and } \tilde c (x)\ge 0 \textrm{ for a.e. } x\in (0,1).    
\]
We introduce the solution $v$ to the elliptic boundary value problem 
\ba
\label{E1} -(\tilde a v_x)_x + \tilde c v &=& 0, \quad x\in (0,1),\\
\label{E2} v(0)=v(1)&=& 1,  
\ea
and set 
\begin{eqnarray}
\label{defu1} u_1(x,t) &:=& e^{-Kt }u(x,t),\\ 
\label{defu2} u_2(x,t) &:=& \frac{ u_1(x,t) }{ v(x) } \cdot 
\end{eqnarray}
Finally, let 
\be
\label{E3}
L:=\int_0^1 (a(s)v^2(s)e^{B(s)} )^{-1}ds, \quad y(x):=\frac{1}{L} \int_0^x (a(s) v^2 (s)e^{B(s)})^{-1} ds 
\ee
and
\be
\label{E4}
\hat u(y,t):=u_2(x,t), \qquad \hat \rho (y):=L^2a(x)v^4(x)e^{2B(x)} \rho (x) 
\ee
for $0<t<T$, $y=y(x)$ with $x\in [0,1]$.  
Then the following result holds.
\begin{proposition}
\label{prop1}
\begin{enumerate}
\item[(i)] $v\in W^{1,1}(0,1)$ and $0<v(x)\le 1$ $\forall x\in [0,1]$;
\item[(ii)] $y:[0,1]\to [0,1]$ is an increasing bijection with $y,y^{-1}\in W^{1,1}(0,1)$;
\item[(iii)] $\hat \rho (y)>0$ for a.e. $y\in (0,1)$, and $\hat \rho \in L^p(0,1)$;
\item[(iv)] $\hat u$ solves the system 
\ba
\hat u_{yy} -\hat\rho \hat u_t &=&0, \quad y\in (0,1),\ t\in (0,T), \label{D1}\\
\hat \alpha _0\hat u(0,t)+\hat \beta _0 \hat u_y(0,t)&=& 0, \quad t\in (0,T), \label{D2}\\
\hat \alpha _1\hat u(1,t)+\hat \beta _1 \hat u_y(1,t)&=&  \hat h(t):=e^{-Kt} h(t), \quad t\in (0,T), \label{D3}\\
\hat u (y(x) ,0)&=& \frac{u_0(x)}{v(x)}, \quad x\in (0,1), \label{D4}
\ea
for some $(\hat \alpha _0,\hat \beta _0),(\hat \alpha _1,\hat \beta _1)\in \R ^2\setminus \{ (0,0) \}$.  
\end{enumerate}

\end{proposition}

\begin{proof}
(i) Let $l=\int_0^1 ds/\tilde a(s)$ and $z(x)=l^{-1} \int_0^x ds/\tilde a(s)$. Then $z:[0,1]\to [0,1]$ is a strictly increasing continuous map 
(for $z(x_2)-z(x_1) =l^{-1}\int_{x_1}^{x_2} ds/\tilde a(s) >0$ for $x_1<x_2$). It is a bijection which is absolutely continuous  (i.e.  $z\in W^{1,1}(0,1)$), for 
$1/ \tilde a \in L^1(0,1)$. Moreover, $z'(x)=1/ (l\tilde a(x))$ for a.e. $x\in (0,1)$. It follows from \eqref{B11} and \eqref{B14} that  $a(x)<\infty$ and $\tilde a(x) <\infty$ for a.e. $x\in (0,1)$, so that 
 $z'(x)>0$ for a.e. $x\in (0,1)$. Then we infer from
a theorem due to M. A. Zareckii (see \cite[Ex. 5.8.54 p. 389]{bogachev} or \cite{spa}) that  $z^{-1}$ is absolutely continuous as well (i.e.
$z^{-1}\in W^{1,1}(0,1)$). (Note that for $z:[0,1]\to [0,1]$ a strictly increasing bijection in $W^{1,1}(0,1)$, its inverse $z^{-1}$  
may not belong to  $W^{1,1}(0,1)$, see \cite[Ex. 4.6 p. 287]{gordon} or \cite{spa}.) 
In particular, $z$ satisfies the condition $N$ (Lusin's condition) 
\be\label{lusin}
A\subset [0,1], \ |A|=0 \Rightarrow |z(A)|=0
\ee
($|A|$ standing for the Lebesgue measure of $A$), and the same holds true for $z^{-1}$. 

Introduce the function $w:[0,1]\to \R$ defined by 
\[
w(z):=v(x(z))\qquad \forall z\in [0,1].
\] 
Then $dw/dz = l\tilde a (x)dv/dx$ so that, letting $'=d/dz$ and $\gamma (z) := (l^2 \tilde a\tilde c)(x(z))$, \eqref{E1}-\eqref{E2} becomes
\ba
\label{E11} -w''+\gamma w &=& 0, \quad z\in (0,1)\\
\label{E12} w(0)=w(1)&=&1.
\ea  
Note that $\gamma (z)\ge 0$ for a.e. $z\in (0,1)$ and that $\gamma \in L^1(0,1)$, for 
\[
\int_0^1\gamma (z) dz = l \int _0^1 \tilde c(x(z)) \frac{dx}{dz} dz= l \int _0^1 \tilde c (x) dx <\infty. 
\]
In the last equality, we used the change of variable formula (which is licit, because $z^{-1}\in W^{1,1}(0,1)$ and it satisfies Lusin's condition, see \cite{hajlasz93}). 
Letting $w=u+1$, we define $u$ as the unique solution in $H^1_0 (0,1)$ of the variational problem 
\[
\int_0^1 [ u'\varphi ' + \gamma u\varphi ] dx = -\int _0^1 \gamma \varphi dx \quad \forall \varphi \in H^1_0(0,1).  
\]
Then $w\in W^{2,1}(0,1)\subset C^1([0,1])$. Let us check that 
\be
\label{E15} 0<w(x) \le 1\quad \forall x\in [0,1]. 
\ee
If $\max_{x\in[0,1]} w(x)>1$, we can pick $x_0\in (0,1)$ such that 
\be
\label{E15bis}
w(x_0)=\max_{x\in [0,1] } w(x) >1. 
\ee
Then $w'(x_0)=0$. Let $\delta >0$ denote the greatest positive number such that $x_0+\delta \le 1$ and
\[
w(x)>1 \qquad \forall x\in (x_0,x_0+\delta ).
\] 
It follows that for $x\in [x_0,x_0+\delta ]$ 
\[
w'(x)=\int_{x_0}^x w''(s)ds =\int_{x_0}^x \gamma (s) w(s) ds\geq0
\]
and hence 
\[
w(x)-w(x_0)=\int_{x_0}^x w'(s)ds\ge 0.
\]
In particular, $w(x_0+\delta )\ge w(x_0) >1$, a fact which contradicts the definition of $\delta$. 
Thus  $\max_{x\in[0,1]} w(x)\le 1$. A similar argument shows that 
 $\min_{x\in[0,1]} w(x)\ge 0$. If  $\min_{x\in[0,1]} w(x)=0$, we pick $x_0\in (0,1)$ such that 
 \[
 w(x_0)=\min_{x\in [0,1] } w(x)=0. 
 \]  
 Then $w$ solves the Cauchy problem 
 \begin{eqnarray*}
&& w''(x)=\gamma (x)w(x)\quad \text{ for a.e.  } x\in  (0,1),\\
&&w(x_0)=w'(x_0)=0
 \end{eqnarray*}
 and hence $w\equiv 0$, which contradicts \eqref{E12}. \eqref{E15} is proved. \\
(ii) $y:[0,1]\to [0,1]$ is an increasing continuous map (for $dy/dx=(L a v^2 e^B)^{-1}>0$ a.e. in $(0,1)$). Moreover, 
$y\in W^{1,1}(0,1)$ (using \eqref{B12} and (i)), and also $y^{-1}\in W^{1,1}(0,1)$. (See (i) for the proof of a similar result for $z$.) \\
(iii) To check that $\hat \rho \in L^p(0,1)$ when $1<p<\infty$, we use \eqref{B14}, \eqref{E3}-\eqref{E4} and (i) to get 
\begin{eqnarray*}
\int_0^1 |\hat \rho (y) |^p dy &=& \int _0^1 [L^2 a(x) v^4 (x) e^{2B(x)} \rho (x) ]^p \frac{dy}{dx} dx\\
&=& L^{2p-1} \int_0^1 a^{p-1} \rho ^p v^{4p-2} e^{(2p-1)B} dx\\
&<&\infty .
\end{eqnarray*}
The fact that $\hat \rho \in L^\infty (0,1)$ when \eqref{B14} holds with $p=\infty$ is obvious.
On the other hand, $\hat \rho (y) >0$ for a.e. $y\in (0,1)$, for 
\[
\int_0^1 \chi_{ \{ \hat \rho (y)\le 0\} } (y)dy = \int_0^1 \chi _{ \{ (a\rho )(x)\le 0 \} }(x) \frac{dy}{dx} dx =0. 
\]
(iv) We first derive the PDE satisfied by $u_2$. 
\ba
e^{-B} (av^2 e^B u_{2,x} )_x 
&=& e^{-B} \big( a v^2 e^B (\frac{u_{1,x}}{v} -\frac{u_1}{v^2} v_x ) \big) _x  \nonumber \\
&=& e^{-B} \big(a e^B (vu_{1,x} -v_x u_1 )\big) _x \nonumber \\
&=& e^{-B} \big( v(a e^B u_{1,x})_x - u_1(a e^B v_x)_x \big) \nonumber \\
&=& v\rho u_{1,t}  \nonumber \\
&=& \rho v^2 u_{2,t}\label{E20} 
\ea
(The first equality follows from \eqref{defu2}, the third from basic algebra, the fourth from \eqref{B1}, \eqref{XXX1}-\eqref{E1} and \eqref{defu1},  and the last from \eqref{defu2} again.)    
Since $\partial _y=(dx/dy)\partial _x = La v^2 e^B \partial _x$, \eqref{E20} combined with \eqref{E4} gives \eqref{D1}. \eqref{D4} is obvious. 
It remains to establish \eqref{D2}-\eqref{D3}. We focus on \eqref{D2}, \eqref{D3} being obtained the same way. From the definition of $u_2$ we obtain
\be
au_x=e^{Kt}a(v_xu_2+vu_{2,x})\quad \textrm{ a.e. in } (0,1). 
\label{E3bis}
\ee
Combined with \eqref{B2}, this gives
\[
\alpha _0 u_2(0,t) + \beta _0 ((av_x) (0) u_2(0,t) + (au_{2,x})(0,t))=0. 
\]
On the other hand 
\[ \hat u_y=(dx/dy) u_{2,x}=La(x)v^2(x)e^{B(x)} u_{2,x}\]
and hence $\hat u_y(0,t)=L(au_{2,x})(0,t)$. Then \eqref{D2} follows with 
\[
\hat \alpha _0=\alpha _0+ \beta _0 (av_x)(0),\quad \hat \beta _0= L^{-1} \beta_0.
\]
\end{proof}
\subsection{Null controllability of the control problem \eqref{C1}-\eqref{C4}}
Assume given $p\in (1,\infty ] $, $\rho \in L^p(0,1)$ with $\rho (x)>0$ for a.e. $x\in (0,1)$, and  
$(\alpha _0,\beta _0),(\alpha _1,\beta _1)\in \R ^2\setminus \{ (0,0) \}$. Let $'=d/dx$, and let
\[
L^2_\rho :=\big\{ f:(0,1)\to \R ;  || f ||^2_{L^2_\rho} :=\int_0^1 |f(x)|^2 \rho (x) dx <\infty \big\} . 
\]
\begin{proposition}
\label{prop2}
Let $p,\rho,\alpha _0,\beta _0,\alpha _1$, and $ \beta _1$ be as above.  Then there are a sequence $(e_n)_{n\ge 0}$ in $L^2_\rho$ and a 
sequence $(\lambda _n)_{n\ge 0}$ in $\R$ such that
\begin{enumerate} 
\item[(i)] $(e_n)_{n\ge 0}$ is an orthonormal basis in $L^2_\rho$;
\item[(ii)] For all $n\ge 0$, $e_n\in W^{2,p}(0,1)$ and $e_n$ solves 
\begin{eqnarray}
-e_n'' &=& \lambda _n \rho e_n \quad \textrm{ in } (0,1), \label{F1}\\
\alpha _0 e_n(0)+\beta _0 e_n'(0)&=& 0, \label{F2}\\
\alpha _1e_n(1) +\beta _1e_n'(1)&=& 0.\label{F3}
\end{eqnarray}  
\item[(iii)] The sequence $(\lambda _n)_{n\ge 0}$ is strictly increasing, and for some constant $C>0$
\be
\label{G100}
\lambda _n\ge Cn\quad \textrm{ for } n\gg 1 .
\ee 
\end{enumerate}
\end{proposition}
\begin{proof}
Let us consider the elliptic boundary value problem 
\begin{eqnarray}
-u'' +\lambda ^* \rho u &=& \rho f \quad \textrm{ in } (0,1),\label{G12}\\
\alpha _0 u(0)+\beta _0 u'(0)&=& 0,\label{G13}\\
\alpha _1u(1) +\beta _1 u'(1)&=& 0\label{G14}
\end{eqnarray}  
where $\lambda ^*\gg 1$ will be chosen later on. Introduce the symmetric bilinear form 
\[
a(u,v) :=\int_0^1  (u'v'+\lambda ^* \rho uv)dx + a_b(u,v)
\]
where
\[
a_b(u,v) := \left\{ 
\begin{array}{ll}
\frac{\alpha _1}{\beta _1} u(1)v(1) - \frac{\alpha _0}{\beta _0} u(0)v(0)  \ \ &\text{ if } \beta _1\ne 0 \textrm{ and } \beta _0\ne 0,\\
\frac{\alpha _1}{\beta _1} u(1)v(1) \ \ &\text{ if } \beta _1\ne 0 \textrm{ and } \beta _0 = 0,\\
- \frac{\alpha _0}{\beta _0} u(0)v(0)  \ \ &\text{ if } \beta _1 =  0 \textrm{ and } \beta _0\ne 0,\\
0 \ \ &  \text{ if } \beta _1 =  0 \textrm{ and } \beta _0= 0.
\end{array}
\right. 
\]
Let 
\[
H :=\{ u\in H^1(0,1); \ u(0)=0 \textrm{ if } \beta_0 =0, \ u(1)=0 \textrm{ if } \beta _1=0\}
\]
be endowed with the $H^1(0,1)$-norm. Clearly, the form $a$ is continuous on $H\times H$, for $H^1(0,1)\subset C^0([0,1])$ continuously. We claim that the form $a$ is 
{\em  coercive}  if $\lambda ^*$ is large enough. We need the 
\begin{lemma}\label{lem1}
For any $\varepsilon >0$, there exists some number $C_\varepsilon>0$ such that
\be
\label{G1}
||u||_{L^\infty}^2 \le \varepsilon || u' ||^2_{L^2} + C_\varepsilon || u ||^2_{L^2_\rho} \qquad \forall u\in H^1(0,1).
\ee 
\end{lemma}
\noindent
{\em Proof of Lemma \ref{lem1}.} If \eqref{G1} is false, then one can find a number $\varepsilon >0$ and a sequence $(u_n)_{n\ge 1}$ in $H^1(0,1)$ such that 
\be
\label{G2} 1=||u_n||^2_{L^\infty } >\varepsilon ||u_n '||^2_{L^2} +n||u_n||^2_{L^2_\rho}\quad  \forall \ge 1. 
\ee 
Thus $||u_n||^2_{H^1} \le 1+\varepsilon ^{-1}$, and for some subsequence $(u_{n_k})$ and some $u\in H^1(0,1)$ we have 
\be
u_{n_k} \to u \textrm { weakly in } H^1(0,1). 
\ee
Since $H^1(0,1)\subset C^0([0,1])\subset L^2_\rho$ continuously, the first embedding being also compact, we infer that $u_{n_k}\to u$ in both 
$C^0([0,1])$ and $L^2_\rho$. Thus $||u||_{L^\infty}=1$ by \eqref{G2}. But \eqref{G2} yields also $u_n\to 0$ in $L^2_\rho$ and hence  $u=0$, 
contradicting $||u||_{L^\infty}=1$. Lemma \ref{lem1} is proved.\qed

From \eqref{G1}, we infer the existence of some constants $C_1,C_2>0$ such that 
\be
C_1 ||u||^2_{H^1} \le ||u' ||^2_{L^2} + ||u|| ^2_{L^2_\rho}  \le C_2 ||u||^2_{H^1}\quad \forall u\in H^1(0,1). 
\ee
Next, we have for some $C^*>0$
\be
|a_b(u,u)| \le C^*||u||^2_{L^\infty} \le C^* (\varepsilon ||u'||^2_{L^2} + C_\varepsilon ||u||^2_{L^2_\rho} )\le \frac{1}{2} (||u'||^2_{L^2} + \lambda ^* ||u||^2_{L^2_\rho})
\ee
if we pick $0<\varepsilon <(2C^*)^{-1}$ and $\lambda ^* > 2C^*C_\varepsilon$. Then for all $u\in H^1(0,1)$ we have 
\[
a(u,u) \ge \frac{1}{2} (||u'||^2_{L^2} + \lambda ^* ||u||^2_{L^2_\rho} ) \ge C ||u||^2_{H^1},
\]
with $C:= \min( 1,\lambda ^*) C_1/2$, as desired.

Let $f\in L^2_\rho$ be given. The linear form $L(v)=\int_0^1\rho fvdx$ being continuous on  $H$, 
it follows from Lax-Milgram theorem that there exists 
a unique function $u\in H$ such that 
\be
\label{G10}
a(u,v)=L(v) \quad \forall v\in H.
\ee 
Taking any $v\in C_0^\infty (0,1)$ in \eqref{G10}, we  infer that \eqref{G12} holds in the distributional sense. Furthermore $u\in W^{2,1}(0,1)$. Next, multiplying 
each term in  \eqref{G12} by $v\in C^\infty ([0,1])\cap H$, integrating over $(0,1)$ and comparing with \eqref{G10}, we obtain \eqref{G13}-\eqref{G14}. 

The operator $T:f\in L^2_\rho\to u=T(f)\in L^2_\rho$ is continuous, compact, and self-adjoint. It is also positive definite, for 
\[
C ||u||^2_{H^1} \le a(u,u)=(f,u)_{L^2_\rho}\quad \text{ and } \quad u=0\iff f=0.
\]  
By the spectral theorem, there are an orthonormal basis $(e_n)_{n\ge 0}$ in $L^2_\rho $ and a sequence $(\mu_n)_{n\ge 0}$ in $(0,+\infty )$ 
with $\mu _n\searrow 0$ such that $T(e_n)=\mu _n e_n$ for all $n\ge 0$. Thus \eqref{F1}-\eqref{F3}  hold with $\lambda _n = \mu _n^{-1}-\lambda ^*$.   The eigenfunction
$e_n\in W^{2,p}(0,1)$ by \eqref{F1} and the fact that $\rho\in L^p(0,1)$ and  $e_n\in L^\infty(0,1)$.\\ 
(iii) The sequence $(\lambda _n)_{n\ge 0}$ is known to be nondecreasing. It is (strictly) increasing if each eigenvalue $\lambda _n$ is simple, a fact which is easily established: 
if $e$ and $\tilde e$ are two eigenfunctions associated with the same eigenvalue $\lambda_n$, then the Wronskian $W(x):=e(x)\tilde e'(x) - e'(x) \tilde e(x)$ satisfies
$W'(x)=0$ a.e. and $W(0)=0$, and hence  $W\equiv 0$ in $(0,1)$. It follows that $e$ and $\tilde e$ are proportional. 

Let us prove \eqref{G100}. Consider for any $\lambda \ge 1$ the system 
\ba
-e''&=& \lambda \rho e, \label{R1}\\
\alpha _0 e(0) + \beta _0 e'(0) &=& 0, \label{R2} \\
\alpha _1 e(1) + \beta _1 e'(1) &=& 0. \label{R3} 
\ea
Following \cite{BR}, we introduce the Pr\"ufer substitution
\ba
e' &=& r\cos \theta , \label{R4}\\
e &=& r  \sin \theta , \label{R5}
\ea  
 so that 
\ba
r^2&=& e'^2 + e^2, \label{R6} \\
\tan \theta &=& \frac{e}{e'} \cdot \label{R7}
\ea
Then $(r,\theta )$ satisfies 
\ba
\frac{dr}{dx} &=& r(1-\lambda\rho )\cos\theta \sin \theta, \label{R12} \\
\frac{d\theta}{dx} &=& \cos ^2 \theta + \lambda\rho \sin ^2 \theta. \label{R13}
\ea
Conversely, if $(r,\theta)$ satisfies \eqref{R12}-\eqref{R13}, then one readily sees that \eqref{R1} and \eqref{R4} hold.

The condition \eqref{R2} is expressed in terms of $\theta$ as 
\be
\label{R14}
\theta _{\vert x=0}=\theta _0  :=
\left\{ 
\begin{array}{ll}
-\arctan ( \frac{\beta _0}{\alpha _0} ) & \text{ if } \alpha _0\ne 0,\\ 
\frac{\pi}{2} &\text{ if } \alpha _0=0. 
\end{array}
\right. 
\ee
Denote by $\theta (x,\lambda )$ the solution of \eqref{R13} and \eqref{R14}. (Note that $r$ is not present in \eqref{R13}.) 
Introduce
\be
\label{R14bis}
\theta _1 :=
\left\{ 
\begin{array}{ll}
-\arctan ( \frac{\beta _1}{\alpha _1} ) & \text{ if } \alpha _1\ne 0,\\ 
\frac{\pi}{2} &\text{ if } \alpha _1=0. 
\end{array}
\right. 
\ee
Then $(e,\lambda)$ is a pair of eigenfunction/eigenvalue if and only if 
\be
\label{R14ter}
\theta ( 1,\lambda)=\theta _1 \quad \text{ mod } \pi.  
\ee
Since the map $(x,\theta,\lambda ) \to \cos ^2 \theta + \lambda\rho (x) \sin ^2 \theta$ is integrable in $x$ and of class $C^1$ in $(\theta , \lambda )$, 
it follows that the map $(x,\lambda ) \to \theta (x,\lambda )$ is well defined and continuous for $x\in [0,1]$ and $\lambda \ge 1$.  On the other hand, since 
 the map   $\lambda  \to \cos ^2 \theta + \lambda\rho (x) \sin ^2 \theta$ is strictly increasing for a.e. $x$ (provided that $\theta\not\in\pi\Z$), it follows from 
a classical comparison theorem (see e.g. \cite{BR})  that the map $\lambda \to \theta (1,\lambda)$ is strictly increasing. 

Let 
\[
\bar \theta (x):=\lim_{\lambda \to \infty} \theta (x, \lambda ),\qquad x\in [0,1].
\]
We claim that 
\be
\label{R18}
\bar\theta (1) =\infty . 
\ee
If \eqref{R18} fails, then we have for all $x\in [0,1]$ and all $\lambda \ge 1$
\[
\theta _0 \le \theta (x, \lambda) \le \theta (1, \lambda ) \le \bar \theta (1)<\infty,
\]
where we used the fact that the r.h.s. of \eqref{R13} is positive a.e. 
Integrating in \eqref{R13} over $(a,b)$, where $0\leq a<b\leq1$, gives then
\be
\theta (b, \lambda ) - \theta (a, \lambda ) =\int_a^b \cos ^2\theta(x, \lambda)  dx +  \lambda\int_a^b\rho (x) \sin ^2 \theta(x, \lambda) dx\label{eq:3}.
\ee
 An application of the Dominated Convergence Theorem yields
\begin{eqnarray}
\int_a^b \cos ^2 \theta(x, \lambda) dx &\to & \int_a^b\cos ^2\bar \theta(x) dx,\label{eq:1}\\
\int_a^b \rho (x)\sin ^2 \theta(x, \lambda) dx &\to& \int_a^b\rho (x)\sin ^2 \bar \theta(x) dx\label{eq:2}
\end{eqnarray}
as $\lambda \to \infty$.

Letting $\lambda\to\infty$ in \eqref{eq:3}
and using~\eqref{eq:1}-\eqref{eq:2}, we infer
\[
\int_a^b\rho (x) \sin ^2 \bar\theta(x) dx = 0.
\]
The numbers $a$ and $b$ being arbitrary,
this shows that $\bar\theta(x)\in\pi\Z$ for a.e.~$x\in(0,1)$. The function $\bar\theta$ being nondecreasing and bounded,
 it is piecewise constant. Choosing $a<b$ such that $\bar\theta$ is constant on $[a,b]$ and letting $\lambda\to\infty$ in~\eqref{eq:3}, we obtain $0\geq b-a$, which is a contradiction.

Thus \eqref{R18} is established, and we see that for any $n\gg 1$ we can find a unique $\tilde \lambda _n\ge 1$ such that
\[
\theta (1, \tilde\lambda _n) =\theta _1 + n\pi.  
\]
Then $\lambda_n$ and $\tilde \lambda _n$ must agree, up to a translation in the indices, i.e. $\lambda _n=\tilde \lambda _{n-\bar n}$ for some $\bar n\in \Z$.  
Thus we can write 
\[
\theta (1, \lambda _n ) =\theta _1 + (n-\bar n)\pi.\  
\] 
Integrating in \eqref{R13}, we obtain 
\[
\theta _1 +(n-\bar n ) \pi -\theta _0 =\int_0^1 (\cos ^2 \theta + \lambda_n\rho \sin ^ 2\theta ) dx 
\le 1+\lambda_n\int_0^1 \rho  (x) dx. 
\]
 Since $\theta _0, \theta _1\in (-\pi/2, \pi/2]$ and $\int_0^1\rho (x)dx>0$, \eqref{G100} follows. \end{proof}
\begin{remark}
If, in addition, $\alpha _0\beta _0\le 0$ and $\alpha _1\beta _1\ge 0$, then using a modified Pr\"ufer system as in 
\cite{BR,harris} we can actually prove that 
\[
\lambda _n\ge Cn^2 \qquad \text{ for } \  n\gg 1. 
\]
\end{remark}

 We now turn our attention to the generating functions $g_i$ ($i\ge 0$) defined along \eqref{C10}-\eqref{C15}. 
  \begin{proposition}
  \label{prop3} \mbox{}
   \begin{enumerate}
  \item[(i)] $g_0(x)=(\alpha _0^2+\beta _0^2)^{-1} (\beta _0 - \alpha _0 x )$\\
  \item[(ii)] There are some constants $C,R>0$ such that 
 \be
 \label{H1} 
 || g_i ||_{ W^{2,p} (0,1) } \le \frac{C}{ R^i ( i ! )^{2-\frac{1}{p} } } \qquad \forall i\ge 0\cdot
 \ee 
  \end{enumerate}
  \end{proposition}
  \begin{proof}
  (i) is obvious. For (ii), we first notice that $g_i$ may be written as 
  \be
  \label{H2}
  g_i(x)=\int_0^x \big( \int_0^s \rho (\sigma ) g_{i-1}(\sigma ) d\sigma \big) ds. 
  \ee 
  Let $q\in [1,\infty )$ be the conjugate exponent of $p$, i.e. $p^{-1}+q^{-1}=1$. We need the following
  \begin{lemma}
  \label{lem2}
  Let $f\in L^\infty (0,1)$ and $g(x)=\int_0^x \big( \int_0^s \rho (\sigma ) f(\sigma ) d\sigma  \big) ds$. If 
  \be
  \label{H3}
  |f(x) | \le Cx^r \textrm{ for a.e. }\   x\in (0,1) 
  \ee 
  for some constants $C,r\ge 0 $, then 
  \be
  \label{H4}
  |g(x) |  \le C\frac{ ||\rho ||_{L^p} }{q^\frac{1}{q} } \frac{ x^{r+\frac{1}{q} +1} }{ (r+\frac{1}{q} )^{\frac{1}{q}} ( r + \frac{1}{q} +1 )} \quad \forall x\in [0,1].
  \ee
  \end{lemma} 
  \noindent
  {\em Proof of Lemma \ref{lem2}.}  From the H\"older inequality and \eqref{H3}, we have for all $s\in (0,1)$ 
  \begin{eqnarray*}
  |\int_0^s \rho (\sigma ) f(\sigma ) d\sigma  | &\le & ||\rho ||_{ L^p(0,s) } ||f ||_{ L^q (0,s)} \\
  &\le& C\|\rho ||_{L^p(0,1)} \left( \frac{s^{rq+1} }{rq+1}\right)^{\frac{1}{q}}  
  \end{eqnarray*} 
  so that 
  \[
  |g(x)| \le C || \rho ||_{ L^p(0,1) } \frac{ x^{ r+\frac{1}{q} +1}   }{ (rq+1)^\frac{1}{q} (r+\frac{1}{q} +1) }
\quad \forall x\in [0,1]. 
  \]
  \qed
  
 Iterated applications of Lemma \ref{lem2} yield
 \begin{eqnarray*}
 |g_i(x)| &\le& || g_0 ||_{L^\infty}  \left( \frac{||\rho ||_{L^p} }{ q^\frac{1}{q} }\right) ^i \frac{x^{i ( \frac{1}{q} +1)  }}{
 \prod_{j=1}^i \left( \frac{1}{q} + (j-1)(1+\frac{1}{q} )\right) ^\frac{1}{q} \prod_{j=1}^i j(1+\frac{1}{q} ) }\\
 &\le& || g_0 ||_{L^\infty}  \left( \frac{||\rho ||_{L^p} }
 { q^\frac{1}{q} }\right) ^i \frac{ 1}{ \left( \frac{1}{q} ( 1+\frac{1}{q} )^{i-1} (i-1)! \right)^\frac{1}{q} i ! (1+\frac{1}{q} )^i }\\
 &\le& \frac{C}{R^i  i! ^{1+ \frac{1}{q} } }
 \end{eqnarray*}
 if we pick $R  < ||\rho ||_{L^p}^{-1}q^\frac{1}{q}  (1+\frac{1}{q})^{1+\frac{1}{q}}$ and $C\gg 1$. Since $1/q=1-1/p$, we infer that 
 \[
 ||g_i ||_{ L^\infty} \le \frac{C}{ R^i i! ^{2-\frac{1}{p} } }
 \]
 which, combined with \eqref{C13}, yields \eqref{H1}. 
  \end{proof}
 \begin{remark}\mbox{}
 \begin{enumerate}
 \item The power of $i!$ in the computations above is essentially sharp, since
 \[
s^i i!\le \prod_{j=1}^i (r+js) \le s^i (i+1)!
 \]
 for $0\le r\le s$. 
 \item When $p=1$, the estimate $||g_i ||_{L^\infty (0,1)} \le C / (R^i \, i!)$ is not sufficient to ensure the convergence of the series in \eqref{C6} when 
 $f\in G^s([0,T])$ with $1<s<2$. 
 \end{enumerate}
 \end{remark}
  
The fact that we can expand the eigenfunctions in terms of the generating functions is detailed in the following 
\begin{proposition}
\label{prop4}
There is some sequence $(\zeta _n)_{n\ge 0}$ of real numbers such that for all $n\ge 0$ 
\be
\label{H10}
e_n=\zeta _n \sum_{i\ge 0} (-\lambda _n)^i g_i \quad \textrm{ in } W^{2,p}(0,1).
\ee
Furthermore, for some constant $C>0$, we have 
\be
|\zeta _n| \le C (1+|\lambda _n| ^\frac{3}{2} )\quad \forall n\ge 0.
\label{G200}
\ee
\end{proposition}  
\begin{proof}
From \eqref{H1}, we infer that the series in \eqref{H10} is absolutely  convergent, hence convergent, in $W^{2,p}(0,1)$. 
Let $\tilde e : =\zeta_n\sum_{i\ge 0} (-\lambda _n)^ig_i$, where $\zeta_n\in \R$. Then 
\[ \tilde e ''=\zeta _n\sum_{i\ge 1} (-\lambda _n)^i\rho g_{i-1} =-\lambda _n \rho \tilde e \textrm{ in } L^p(0,1),\]
where we used \eqref{C10} and \eqref{C13}. \eqref{C11} and \eqref{C14}-\eqref{C15} yield 
\[
\alpha _0 \tilde e (0) + \beta _0 \tilde e'(0) =0.
\]
On the other hand, using \eqref{C12} and \eqref{C14}-\eqref{C15}, we obtain 
\[
\beta _0 \tilde e (0) -\alpha _0 \tilde e'(0)= \zeta _n \big( \beta _0 g_0(0)  - \alpha _0 g_0' (0) \big) =\zeta _n . 
\]
Hence, if we pick 
\be
\zeta _n := \beta _0 e_n(0) - \alpha _0e_n' (0),
\label{K01}
\ee
 we have that $E:=e_n - \tilde e$ satisfies
\[
\alpha _0 E(0)+\beta _0 E'(0)=\beta _0 E(0) -\alpha _0 E'(0) =0
\] 
and hence $E(0)=E'(0)=0$ which, when combined with $-E'' = \lambda _n \rho E$, yields $E\equiv 0$, i.e. $e_n =\tilde e $.
Thus \eqref{H10} holds with $\zeta _n$ as in \eqref{K01}. To estimate $\zeta _n$, we remind that $e_n$ satisfies $T(e_n)=\mu _ne_n$, and hence
\[
\mu_n a(e_n,e_n) = \int_0^1 \rho \, e_n^2 dx =1.
\]
Since $a(e_n,e_n)\ge C ||e_n||^2_{H^1}$, we infer that $||e_n|| ^2_{H^1} \le C\mu_n^{-1}$, and hence 
\[
|e_n(0)|  + |e_n(1) | \le C|| e_n ||_{ H^1} \le C (1+|\lambda _n|^\frac{1}{2}). 
\]
On the other hand,  \eqref{F1} yields
\[
|| e_n'' ||_{L^p}\le C| \lambda _n|\,  ||\rho ||_{L^p} ||e_n||_{H^1} \le C ( 1+ | \lambda _n |^\frac{3}{2}).
\] 
Thus 
\[
|\zeta _n| \le C || e_n ||_{ W^{2,p} } \le C (1 + | \lambda _n |^\frac{3}{2} ). 
\]
\end{proof}
Since $p>1$, for any $s\in (1, 2-\frac{1}{p} )$ and any $0<\tau <T$, one may pick a function $\varphi \in G^s([0,2T])$ such that 
\[
\varphi (t) =\left\{ 
\begin{array}{ll}
1&\textrm { if } t \le \tau ,\\
0 &\textrm{ if } t\ge T. 
\end{array}
\right. 
\] 
We are in a position to prove the null controllability of \eqref{C1}-\eqref{C4}. Let $u_0\in L^2_\rho$. 
Since $(e_n)_{n\ge 0}$ is an orthonormal basis in $L^2_\rho$, we can write 
\be
\label{K5} 
u_0=\sum_{n\ge 0} c_n \, e_n\quad \textrm{ in } L^2_\rho
\ee
with $\sum_{n\ge 0} |c_n|^2 <\infty$. 
Let 
\be
\label{K6}
y(t) : =\varphi (t) \sum_{n\ge 0} c_n\zeta _n e^{-\lambda _nt}\qquad\text{ for }  t\in [\tau , T]
\ee
and 
\be
\label{K7} 
u(x,t) =
\left\{  
\begin{array}{ll} 
\sum_{n\ge 0} c_n e^{-\lambda _n t} e_n(x) \quad &\textrm{ if } 0\le t\le \tau,\\
\sum_{i\ge 0} y^{ (i) } (t) g_i(x) \quad &\textrm{ if } \tau  <  t\le T.  
\end{array}
\right.
\ee
The main result in this section is the following
\begin{theorem}
\label{thm3} Let $p\in (1,\infty ]$, $\rho \in L^p(0,1)$ with $\rho (x) >0 $ for a.e. $x\in (0,1)$, $T>0$, $\tau \in (0,T)$, and $(\alpha _0,\beta _0),
(\alpha _1,\beta _1) \in \R ^2 \setminus \{ (0 , 0 ) \} $. 
Let $u_0\in L^2_\rho$ be decomposed as in \eqref{K5}, let $s\in (1, 2-1/p)$,  and let $y$ be as in \eqref{K6}. Then $y\in G^s([\tau , T])$, and the control 
\be
\label{K8}
h(t) =
\left\{  
\begin{array}{ll} 
0 \quad &\textrm{ if } 0 \le t\le \tau,\\
\sum_{i\ge 0} y^{ (i) } (t) ( \alpha _1 g_i (1) + \beta _1 g_i ' (1) )  \quad &\textrm{ if } \tau <  t\le T.  
\end{array}
\right.
\ee  
is such that the solution $u$ of \eqref{C1}-\eqref{C4} satisfies $u(.,T)=0$. 
Moreover $ u$ is given by \eqref{K7}, $h\in G^s([0,T])$, and  
$u\in C([0,T],L^2_\rho )\cap G^s ( [\varepsilon ,T], W^{2,p}(0,1) )$ for all $0<\varepsilon \le T$.   
\end{theorem}
\begin{proof}
Let $\C _+:=\{ z=t+ir;\  t>0, \  r\in \R \}$. We notice that the map $z\to \sum_{n\ge 0} c_n\zeta _n e^{-\lambda _n z}$ is analytic  in $ \C_+$. Indeed, by \eqref{G100} and \eqref{G200}, 
the series is clearly uniformly convergent
on any compact set in $\C_+$. It follows that the map $t\to \sum_{n\ge 0} c_n\zeta _n e^{-\lambda _n t}$ is (real) analytic in $(0,\infty)$, hence in $G^1([\tau,T])\subset G^s([\tau,T])$. Thus 
$y\in G^s([\tau ,T])$ by a classical result (see e.g. \cite[Theorem 19.7]{rudin}). 

Let $\bar u$ denote the function defined in the r.h.s. of \eqref{K7}. We first prove that $\bar u \in G^1( [\varepsilon, \tau ],$ $  W^{2,p}(0,1))$ for 
all $\varepsilon \in (0,\tau )$. We have for $k\in \N$ and $\varepsilon \le t\le \tau$, 
\begin{eqnarray*}
||\partial _t ^k(c_ne^{-\lambda _n t} e_n)||_{W^{2,p}} 
&=& |c_n | \, | \lambda _n |^k e^{-\lambda _n t} ||e_n||_{ W^{2,p} } \\
&\le & C |c_n| (1+ | \lambda _n| ^{k+\frac{3}{2}} ) e^{- | \lambda _n|\varepsilon } \\
&\le & C\frac{|c_n|}{n+1} (1+ | \lambda _n| ^{k+3} )e^{- | \lambda _n|\varepsilon } \\
&\le &  C\frac{ |c_n| }{n+1} (1 + \varepsilon ^{-k-3} (k+3)!),
\end{eqnarray*}
where we used \eqref{G100} and $x^k/k! \le e^x$ for $x>0$ and $k\in \N$. Thus, applying Cauchy-Schwarz inequality, we obtain 
for $k\in \N$, $\varepsilon \le t\le \tau$ and some $C,\delta >0$
\[
||\partial _t ^k u ||_{W^{2,p}} \le \sum_{n\ge 0} ||\partial _t ^k (c_n e^{-\lambda _n t}  e_n)||_{W^{2,p}} \le \frac{C}{\delta ^k} k!
\]
which gives that $\bar u \in G^1 ( [\varepsilon , \tau ] ,W^{2,p} (0,1) ) $. It is clear that $\bar u \in C( [ 0,\tau ], L^2_\rho )$. Let us check that 
$\bar u(x,\tau ^- )=\bar u(x, \tau ^+)$. We have that for all $x\in [0,1]$
\ba
\bar u(x,\tau ^-)&=&\sum_{n\ge 0} c_n e^{-\lambda _n \tau } e_n (x) \nonumber \\ 
&=& \sum_{n\ge 0} c_n e^{-\lambda _n \tau } \zeta _n \sum_{i\ge 0} (-\lambda _n)^i g_i(x) \label{K21} \\
&=& \sum_{i \ge 0} \big( \sum_{n\ge 0} c_n \zeta _n e^{-\lambda _n \tau } (-\lambda _n )^i \big) g_i(x) \label{K22} \\
&=& \sum_{i\ge 0} y^{ (i) } (\tau ) g_i(x) \label{K23} \\
&=& \bar u (x,\tau ^+).    \nonumber
\ea
 For \eqref{K21} we used \eqref{H10}. For \eqref{K22}, we used Fubini's theorem for series, which is licit for 
 \begin{eqnarray*}
 \sum_{i,n\ge 0} |c_n\zeta _n e^{-\lambda _n \tau} \lambda _n ^i g_i(x)| 
 &\le& C\sum_{i,n\ge 0} \frac{|c_n|}{n+1} (1+ |\lambda _n | ^{i+3 } ) \frac{e^{- |\lambda _n |\tau} }{ R^i i! ^{ 2-\frac{1}{p} } } \\
&\le &  C\sum_{i,n\ge 0} \frac{|c_n|}{n+1} (1+ \tau ^{-i-3} (i+3) ! ) \frac{ 1 }{ R^i i! ^{ 2-\frac{1}{p} } } \\
 &\le &  C( \sum_{n\ge 0} \frac{|c_n|}{n+1}) ( \sum_{i\ge 0} \frac{1+ \tau ^{-i-3} (i+3) ! }{ R^i i! ^{ 2-\frac{1}{p} } }) \\
 &<&\infty . 
 \end{eqnarray*}   
 Finally for \eqref{K23}, we just used the fact that $\varphi (\tau )=1$ and $\varphi ^{ (i) } (\tau )=0$ for $i\ge 1$. 
 It remains to prove that $\bar u\in G^s( [\tau , T], W^{2,p} (0,1) )$. Since $y\in G^s([\tau , T] )$, there are some constants 
 $C,\rho >0$ such that $ | y^{ (i) } (t)|\le C (i !)^s/\rho ^i$. It follows that for $t\in [\tau, T]$
 \ba
 \sum_{i\ge 0} ||\partial _t^j \, [y^{(i)} (t) g_i  ] \, ||_{W^{2,p}} & = &
 \sum_{i\ge 0} ||y^{ ( i+j ) } (t) g_i  ||_{W^{2,p} } \nonumber \\
 &\le & C\sum_{i\ge 0 }\frac{ (i+j) !^s  }{\rho ^{i+j} } \frac{1}{R^i i ! ^{2-\frac{1}{p} }} \nonumber\\
 &\le & C(\frac{2^s}{\rho })^j \left( \sum_{i\ge 0} ( \frac{2^s}{\rho R} )^i   \frac{1}{i ! ^{2-\frac{1}{p}  -s }}  \right) j ! ^s \label{K30}
 \ea
    where we used $(i+j)! \le 2^{i+j} i! j!$. Note that the series converges in \eqref{K30}, since $s<2-\frac{1}{p}$.
    Thus $\bar u\in G^s([\tau , T] , W^{2,p}(0,1))$. It is clear that \eqref{C1} is satisfied by $\bar u$ in the distributional sense 
    in  $(0,1)\times (0,\tau)$ and in $(0,1)\times (\tau , T)$. In particular
    \[
    \partial _t ^j {\bar u} (x, \tau ^+) =(\rho ^{-1} \partial _x ^2 ) ^j {\bar u} (x, \tau ^+)
    =(\rho ^{-1} \partial _x ^2)^j {\bar u} (x,\tau ^-) =\partial _t ^j {\bar u} (x, \tau ^-) , 
    \]  
    for the two series in \eqref{K7} coincide at $t=\tau$, hence so do their space derivatives. This shows that $\bar u\in G^s( [\varepsilon, T],W^{2,p}(0,1))$ for all
    $\varepsilon \in (0,\tau )$, and that \eqref{C1} holds  for $\bar u$ in  $(0,1)\times (0,T)$. 
    
    The function $h$ defined in \eqref{K8} satisfies \eqref{C3} (with $u$ replaced by $\bar u$), and hence $h\in G^s( [0,T] )$ (for $\bar u\in G^s ([\varepsilon , T],$  $W^{2,p}(0,1) ) $ and 
    $W^{2,p}(0,1)\subset C^1([0,1])$). \eqref{C2} and \eqref{C4} are clearly satisfied by $\bar u$, and hence the solution $u$ of \eqref{C1}-\eqref{C4} is $\bar u$. 
    Finally $u(.,T)=0$, for $y^{(i)} (T)=0$ for all $i\ge 0$.  
\end{proof}
\subsection{End of the proof of Theorem \ref{thm1}}
Let $a,b,c,\rho ,K,p,\alpha _0,\beta _0,\alpha _1,\beta _1,T$, and $\tau$ be as in the statement of Theorem \ref{thm1}. Pick any 
$u_0\in L^1_\rho$ and any $s\in (1,2- 1/p)$. Let $u$ denote the solution of 
\eqref{B1}-\eqref{B4} for a given $h\in G^s([0,T])$. Define $v,y,\hat \rho $, and $\hat u(y,t)$ as in Section \ref{section21}.   Then $\hat u$ 
solves    \eqref{D1}-\eqref{D4} 
 with initial state $\hat u_0(y(x))=u_0(x)/v(x)$.
It may occur that $\hat u_0\not\in L^2_{\hat\rho}$. However, $\hat u_0\in L^1_{\hat \rho }$, for 
 \[
 \int_0^1 |\hat u_0(y)| \hat \rho (y) dy = \int_0^1 |\hat u_0 (y(x)) |\hat \rho (y(x)) | \frac{dy}{dx} | dx = L \int_0^1 |u_0(x)| v(x) e^{B(x)} \rho (x) dx <\infty. 
 \] 
From the proof of Lemma \ref{lem1}, we know that the bilinear form $a(u,v)$ is a scalar product in $H$ whose induced Hilbertian norm is equivalent 
to the usual $H^1$-norm, so that $H$ can be viewed as a Hilbert space for this scalar product. Then it is easy to see that 
\begin{enumerate}
\item[(i)] $(\sqrt{\mu _n} e_n)_{n\ge 0}$ is an orthonormal basis in $H$; \\
\item[(ii)] If, for $a\in \R$, ${\mathcal H} ^a$ denotes the completion of $\text{Span}(e_n; \ n\ge 0)$  for the norm 
\[
||\sum_{n\ge 0} c_ne_n||_a := \left( \sum _{n\ge 0} \mu _n^{-a} |c_n|^2 \right) ^{\frac{1}{2}}, 
\] 
then ${\mathcal H}^0=L^2_{\hat \rho} $ and ${\mathcal H}^1= H $;
\item[(iii)] Identifying $L^2_{\hat\rho}$ with its dual, we obtain the diagram
\[
{\mathcal H}^1=H \subset L^2_{\hat\rho} = (L^2_{\hat\rho} ) ' \subset  H' = {\mathcal H}^{-1}.
\]
\end{enumerate} 
See e.g. \cite[pp. 7-17]{komornik} for details. 
 Since for any $w\in H\subset L^\infty (0,1)$,  
  \[
  \int_0^1 |\hat u_0(y) w(y)| \hat \rho (y) dy \le || w ||_{L^\infty} \int_0^1 |\hat u_0(y)| \hat \rho (y) dy \le C || w ||_H \int_0^1 |u_0(x)| \rho (x) dx, 
  \]
  we infer that $\hat u_0\in H'$. Setting $c_n:=\int_0^1 \hat u_0(y) e_n(y) \hat \rho (y) dy$ for $n\ge 0$, the series $\sum_{n=0}^\infty c_n e^{-\lambda _n t }e_n$ 
  belongs to $C([0,T], H')\cap C((0,\tau ], L^2_\rho )$ and it takes the value $\hat u_0$ at $t=0$. The solution $\hat u$ defined in   \eqref{K7}  with $y$ as in \eqref{K6} 
  solves \eqref{D1}-\eqref{D4} with the control input $\hat h(t)$ defined in \eqref{K8}.  Then the pair
  \begin{eqnarray}
  u(x,t) &:=& e^{Kt} v(x) \hat u( y(x) , t), \\
  h(t) &:=& e^{Kt} \hat h (t)
  \end{eqnarray}
  satisfies \eqref{B1}-\eqref{B4} and $u(.,T)=0$. Pick any $\varepsilon \in (0,T)$. Since $v,y\in W^{1,1} (0,1)$, $\hat u \in G^s ([\varepsilon , T], W^{2,p}(0,1))$,  and
  $\hat h \in G^s ( [ 0,T] )$, we have that 
  \[
  u\in G^s ([ \varepsilon , T], W^{1,1}(0,1) ), \quad h\in G^s( [0,T] ).
  \]
  Finally, since by \eqref{E3} and \eqref{E3bis} we have
  \[
  \tilde a u_x = e^{Kt} \big(  (\tilde a v_x) \hat u (y(x),t) + (Lv) ^{-1} \hat u_y (y(x) , t) \big) 
  \]
  and since $(Lv)^{-1},\tilde a v_x\in W^{1,1} (0,1)$ and  $\hat u \in G^s ( [ \varepsilon , T], W^{2,p}(0,1))$, it follows that 
  \[
  \tilde a u_x, au_x \in G^s ( [ \varepsilon , T ] , C^0 ( [0,1] ))
  \]
  and that \eqref{B2}-\eqref{B3} are satisfied. The proof of Theorem \ref{thm1} is complete.\qed
 \begin{remark}
 Since the map $x\to y(x)$  is absolutely continuous and {\em strictly increasing} on $[0,1]$, and the map
$y\to \hat u_y(y,t)$ is absolutely continuous on $[0,1]$ for all $t\in (0,T]$, we infer that $x\to \hat u_y (y(x),t)$ is absolutely continuous on $[0,1]$
for all $t\in (0,T]$. (See \cite[Ex. 5.8.59 p. 391]{bogachev}.) Thus $a u_x(.,t)\in W^{1,1}(0,1)$ for all $t\in (0,T]$. 
 \end{remark}
  
\section{Proof of Theorem \ref{thm2}}
We shall show that the first step in the proof of Theorem \ref{thm1} (see Section \ref{section21}) can be slightly modified to reduce  
\eqref{B41}-\eqref{B44} to the canonical form \eqref{C1}-\eqref{C4}. Next, the conclusion of Theorem \ref{thm2} will follow from 
Theorem \ref{thm3}.  We distinguish two cases: (i) $0\le \mu <1/4$ (subcritical case) and (ii) $\mu =1/4$ (critical case). \\
(i) Assume that $0\le \mu <1/4$. We relax \eqref{E1}-\eqref{E2} to the problem 
\begin{eqnarray}
v_{xx} + \frac{\mu }{x^2} v = 0,&& x\in (0,1), \label{X1} \\
v(x)>0,&& x\in (0,1), \label{X2}\\
v^{-2}\in L^1(0,1).&& \label{X3}
\end{eqnarray}
The general solution of \eqref{X1} is found to be 
\[
v(x)= C_1 x^{r_1} + C_2 x^{r_2} 
\]
where $C_1,C_2\in\R$ are arbitrary constants, and $r_1,r_2$ denote the roots of the equation $r^2-r+\mu=0$, namely
\[
r_1=\frac{1-\sqrt{1-4\mu }}{2} \in (0,\frac{1}{2}), \quad r_2=\frac{ 1 +  \sqrt{1-4\mu }}{2} \in (\frac{1}{2}, \infty ). 
\]
Then $v(x):=x^{r_1}$ satisfies \eqref{X1}-\eqref{X3}.

From \eqref{B41}, we have that $\tilde a= a\equiv 1$, $B\equiv 0$. We set $u_1:=u$, 
\[
u_2(x,t): =\frac{u(x,t) }{ v(x) }, \quad L:=\int_0^1 v^{-2}(s)ds <\infty, \quad y(x):= L^{-1} \int_0^x v^{-2}(s) ds, 
\]
and 
\[
\hat u(y,t) := u_2 (x,t), \quad \hat \rho (y) :=L^2 v^4(x). 
\]
Again, $y:[0,1]\to [0,1]$ is an increasing bijection with $y,y^{-1}\in W^{1,1}(0,1)$, and $\hat u$ satisfies
\begin{eqnarray}
\hat u_{yy} -\hat \rho (y) \hat u _t&=&0, \quad y\in (0,1),\ t\in (0,T), \label{V1}\\ 
\hat u(0,t)&=&0,\quad t\in (0,T),\label{V2}\\ 
(\alpha _1 + \beta _1 r_1) \hat u(1,t) + \frac{\beta_1}{L} \hat u_y (1,t) &=& h(t),\quad t\in (0,T), \label{V3}\\
\hat u_0(y,0)&=&\hat u_0(y) := \frac{u_0(x)}{v(x)}, \quad y\in (0,1) \cdot \label{V4} 
\end{eqnarray} 
 Note that $\hat u_0\in L^2_{\hat \rho}$, for
 \[
 \int_0^1 |\hat u_0(y)|^2 \hat \rho (y) dy = L\int_0^1|u_0(x)|^2 dx <\infty.
 \] 
 On the other hand $\hat \rho \in L^\infty (0,1)$. 
 By Theorem \ref{thm3}, there is some $h\in G^s( [0,T] )$ such that the solution $\hat u$ of \eqref{V1}-\eqref{V4} satisfies
 $\hat u(.,T)=0$. Furthermore
  \be
  \hat u\in G^s([\varepsilon , T],W^{2,\infty} (0,1)). \label{V10}
  \ee
 The corresponding trajectory $u$ satisfies \eqref{B41}-\eqref{B44}  and $u(.,T)\equiv 0$. 
 Finally,  from the expressions 
 \begin{eqnarray*}
 u &=& v u_2 = v \hat u (y(x), t) \\ 
 u_x&=&  v_x \hat u (y (x), t)  + v\hat u_y (y(x),t) \frac{dy}{dx},
 \end{eqnarray*}
 \eqref{V10}, and the explicit form of $v$, we readily see that $u \in G^s([\varepsilon , T] ,W^{1,1}(0,1))$  and 
 $ x^r u_x\in G^s ( [\varepsilon , T] , C^0 ( [0,1] ) )$  for
 $r>1-r_1=(1+\sqrt{1-4\mu })/2$ and $\varepsilon \in (0,1)$. \\
(ii) Assume now that $\mu =1/4$. Assume first that $\beta _1=0$. We notice that the general solution of \eqref{X1} takes the form 
 \[
 v(x)=C_1 \sqrt{x}\ln x + C_2 \sqrt{x}. 
 \]
 Picking $v(x):=-\sqrt{x}\ln x$, we see that \eqref{X1}-\eqref{X3} are satisfied. Performing the same change of variables as in (i) (but with the new expression of $v$) and
 applying again Theorem \ref{thm3}, we infer the existence of  $h\in G^s( [0,T] )$ such that the solution $\hat u$ of \eqref{V1}-\eqref{V4} satisfies
 $\hat u(.,T)=0$. The corresponding trajectory $u$ satisfies \eqref{B41}-\eqref{B44}  and $u(.,T) = 0$. 
 Furthermore, $u\in G^s ([\varepsilon, T],W^{1,1}(0,1))\cap C^\infty ([\varepsilon,1]\times [\varepsilon, T])$ (by using classical regularity results).
 For the general Robin-Neumann condition at $x=1$ it is sufficient to set $h(t):=\alpha _1 u(1,t)+ \beta _1 u_x(1,t)$ with the trajectory $u$ constructed above with 
 the Dirichlet control at $x=1$. The proof of Theorem \ref{thm2} is complete.     \qed

As a possible application, we consider the boundary control by the flatness approach 
of radial solutions of the heat equation in the ball $B(0,1)\subset \R ^N$ ($2\le N\le 3$). 
Using the radial coordinate $r=|x|$, we thus consider the system
\ba
u_{rr}+\frac{N-1}{r}u_r-u_t&=&0, \quad r\in (0,1),\ t\in (0,T), \label{W1}\\
u_r(0,t)&=&0,\quad t\in (0,T), \label{W2}\\
\alpha _1 u(1,t)+\beta _1 u_r(1,t)&=&h(t),\quad t\in (0,T) \label{W3} \\
u(r,0)&=&u_0(r), \quad r\in (0,1).\label{W4} 
\ea 
Note that Theorem \ref{thm1} cannot be applied directly to \eqref{W1}-\eqref{W4}, for \eqref{B13} fails. (Note that, in sharp contrast, the control on a 
ring-shaped domain $\{ r_0<|x|<r_1 \}$ with $r_1>r_0>0$ is fully covered by Theorem \ref{thm1}, the coefficients in \eqref{W1} being then smooth and bounded.)
 
We use the following change of variables from \cite{colton} which allows to remove the term with the first order derivative in $r$  in \eqref{W1}:
\be
u(r,t)=\tilde u(r,t) \exp (-\frac{1}{2}\int_0^r\frac{N-1}{s} ds) =\frac{\tilde u(r,t)}{r^{\frac{N-1}{2}}}\cdot
\ee 
Then \eqref{W1} becomes 
\be
\tilde u _{rr} + \frac{ (N-1)(3-N) }{4}  \frac{\tilde u}{r^2} -\tilde u _t=0.
\label{WW1}
\ee 
This equation has to be supplemented with the boundary/initial conditions
\ba
\tilde u(0,t)&=&0,\quad t\in (0,T),\label{WW2} \\
(\alpha _1 -\frac{N-1}{2}\beta _1) \tilde u (1,t)+\beta _1 \tilde u _r(1,t) &=& h(t),\quad t\in (0,T), \label{WW3} \\
\tilde u(r,0) &=& r^{\frac{N-1}{2}} u_0(r), \quad r\in (0,R). \label{WW4}
\ea
For $N=3$, \eqref{WW1} reduces to the simple heat equation $\tilde u_{rr}-\tilde u_t=0$ to which Theorem \ref{thm1} can be applied.  
In particular $\tilde u\in G^s([\varepsilon ,T],W^{2,\infty}(0,1))$.
Actually, it is well known that $\tilde u\in C^\infty ([0,1]\times [\varepsilon, T])$, so that we can write a Taylor expansion
\[
\tilde u(r,t) = r\tilde u_r(0,t) + \frac{r^3}{6}\tilde u_{rrr}(0,t) + O(r^4),
\]
where we used the fact that $\tilde u(0,t)=\tilde u_{rr}(0,t)=0$. This yields
\[
u_r(0,t)= \frac{r}{3}\tilde u_{rrr}(0,t)+ O(r^2),
\]
so that \eqref{W2} is fulfilled. \\
For $N=2$, \eqref{WW1}-\eqref{WW4} is of the form \eqref{B41}-\eqref{B44} with $\mu=1/4$.  Therefore 
Theorem \ref{thm2} can be applied to \eqref{WW1}-\eqref{WW4}. Our concern now is the derivation of 
\eqref{W2} when going back to the original variables. Recall that 
\[
u(r,t)=\frac{\tilde u (r,t)}{\sqrt{r}},\quad v(r)=-\sqrt{r}\ln r, \quad y(r)=L^{-1}\int_0^r v^{-2}(s)ds,\quad \hat u(y,t)=\frac{\tilde u(r,t)}{v(r)}=- \frac{u(r,t)}{\ln r},
\]  
so that, with $dy/dr= (Lr\ln ^2 r)^{-1}$,
\[
u_r=-\frac{1}{r} \hat u -\frac{1}{Lr \ln r} \hat u_y. 
\]
This yields at fixed $t\in (0,T)$
\[
\int_0^1(|u|^2+|u_r|^2) rdr \le \int_0^1 r\ln ^2 r |\hat u (y(r))|^2dr + C\int_0^1 \frac{\hat u^2}{r}dr + \int_0^1 \frac{|\hat u_y (y(r))|^2 }{L^2r\ln ^2r} dr=:I_1+I_2+I_3.
\]
Since $\hat u(.,t)\in W^{2,\infty}(0,1)$, both $I_1$ and $I_3$ are finite. On the other hand, using $\hat u(0,t)=0$, we obtain $|\hat u(y(r)  ,t)|\le Cy(r)=\frac{C}{|\ln r |}$, and hence $I_2<\infty$. 
Thus $\int_0^1(|u|^2+|u_r|^2) rdr <\infty$, while for $p\in (2,\infty )$
\[
\int_0^1 |u_{rr} + \frac{u_r}{r}|^p rdr = \int_0^1 |u_t|^p rdr= \int_0^1 |\hat u_t (y(r),t)|^p r|\ln r|^p dr<\infty . 
\]
Thus the function $x\to u(|x|,t)$ belongs to $W^{2,p}(B(0,1))\subset C^1(\overline{B(0,1)})$, so that  \eqref{W2} is satisfied.

\section{Appendix: proof of Corollary \ref{cor-intro}}
We apply first a reduction to a canonical form similar to \eqref{C1}-\eqref{C4} by doing exactly the same changes of variables as those described in 
Section \ref{section21}. With $u_1,u_2,y,\hat u$, and $\hat \rho$ defined as in \eqref{defu1}-\eqref{E4}, 
we infer from \eqref{B100} that 
\[
e^{-B} (av^2 e^B u_{2,x})_x= \rho v^2 u_{2,t} +ve^{-Kt} \chi_\omega f.
\]
Multiplying each term in the above equation by $L^2av^2e^{2B}$, and using the fact that $\partial _y = L a v^2 e^B \partial _x$, we arrive to 
\[
\hat u_{yy}=\hat \rho (y) \hat u_t + \chi_{\hat \omega} \hat f, 
\]
where $\hat \omega = (\hat l_1,\hat l_2) : =(y(l_1), y(l_2))$ and 
\[
\hat f(y(x),t) : =L^2 a(x) v^3(x) e^{2B(x)} e^{-Kt} \chi _\omega (x) f(x,t).
\]
Let $\hat u_0(y(x)) :=u_0(x)/v(x)$. 
Pick $\hat l_1',\hat l_2'$ such that  $\hat l_1<\hat l_1'<\hat l_2'<\hat l_2$, and a function 
$\varphi \in C^\infty ( [0,1] )$ such that $\varphi (y)=1$ for $0\le y\le \hat l_1'$ and $\varphi (y)=0$ for $\hat l_2'\le y\le 1$.  Applying Theorem \ref{thm1}, 
we can find two functions $h^1,h^2\in G^s([0,T])$ such 
that the solutions $\hat u^1,\hat u^2$ of the following systems
  \ba
\hat u^1_{yy} -\hat \rho (y) \hat u^1_t&=&0, \quad y\in (0,1),\ t\in (0,T), \label{B101}\\
\hat \alpha _0 \hat u^1(0,t)+ \hat \beta _0 \hat u^1_y(0,t)&=& 0, \quad t\in (0,T), \label{B201}\\
\hat u^1_y(1,t)&=& h^1(t), \quad t\in (0,T), \label{B301}\\
\hat u^1(y,0)&=& \hat u_0(y), \quad y\in (0,1), \label{B401}
\ea
and
  \ba
\hat u^2_{yy} -\hat \rho (y) \hat u^2_t&=&0, \quad y\in (0,1),\ t\in (0,T), \label{B102}\\
\hat u^2_y(0,t)&=& h^2(t), \quad t\in (0,T), \label{B202}\\
\hat \alpha _1\hat u^2(1,t)+\hat \beta _1 \hat u^2_y(1,t)&=&0 , \quad t\in (0,T), \label{B302}\\
\hat u^2(y,0)&=& \hat u_0(y), \quad y\in (0,1), \label{B402}
\ea
satisfy
\[
\hat u^1(y,T)=\hat u^2(y,T)=0 \ \textrm{ for all }\  y\in [0,1]. 
\]
Then it is sufficient to set 
\begin{eqnarray}
\hat u (y,t)  &:=&\varphi (y) \hat u^1 (y,t) + (1-\varphi (y) ) \hat u^2(y,t), \label{ABCDE1}\\
\hat f(y,t) &: =& \varphi '' (y) (\hat u^1(y,t)-\hat u^2(y,t) ) +2\varphi ' (y) (\hat u^1_y (y,t)  - \hat u^2_y (y,t) ). \label{ABCDE2} 
\end{eqnarray}
Note that $\hat f$ is supported in $[\hat l_1',\hat l_2']\times [0,T]$, with $\hat f\in G^s([\varepsilon, T],W^{1,1}(0,1))$ for all $\varepsilon \in (0,T)$, 
and that $\hat u$  solves 
 \ba
\hat u_{yy} -\hat \rho (y) \hat u_t&=&\chi _{\hat \omega} \hat f, \quad y\in (0,1),\ t\in (0,T), \label{B103}\\
\hat \alpha _0 \hat u(0,t)+ \hat \beta _0 \hat u_y(0,t)&=& 0, \quad t\in (0,T), \label{B203}\\
\hat \alpha _1 \hat u(1,t)+ \hat \beta _1 \hat u_y(1,t)&=& 0, \quad t\in (0,T), \label{B303}\\
\hat u(y,0)&=& \hat u_0(y), \quad y\in (0,1), \label{B403}\\
\hat u(y,T)&=& 0, \quad y\in (0,1). \label{B503}
\ea
Let 
\[
f(x,t):= \big( L^2 a(x) v^3(x) e^{2B(x)} e^{-Kt} \big) ^{-1} \hat f (y(x),t). 
\]
Then $f$ is supported in $[y^{-1}(\hat l_1'), y^{-1} (\hat l_2')] \times [0,T] \subset \omega \times [0,T]$. 
We claim that  $f\in L^2(0,T, L^2_a(\omega) )$. Indeed,  we have that
\begin{eqnarray*}
\int_0^T \!\!\! \int_\omega | f(x,t) | ^2 a(x) dxdt 
&\le& C \int_0^T \!\!\! \int_0^1 \chi_\omega (x) ( L^3 a(x) v^4(x) e^{3B(x)} e^{-2Kt} ) |f(x,t)|^2 dxdt \\
&=&  C\int_0^T\!\!\! \int_0^1 \chi _{\hat \omega} (y(x)) |\hat f ( y(x), t) |^2 \vert \frac{dy}{dx }\vert dxdt \\
&= &  C \int_0^T\!\!\! \int_{\hat\omega}   |\hat f(y,t) |^2 dydt ,
\end{eqnarray*} 
and the last integral is finite, since  $\hat f$ is given by \eqref{ABCDE2} and
$\hat u^1, \hat u^2\in L^2(0,T,H^1(0,1))$.
For $\hat u^1$, this can be seen by scaling \eqref{B101} by $\hat u^1$, integrating over $(0,1)\times (0,t)$
for $0<t\le T$, and using
Gronwall's lemma combined with Lemma \ref{lem1}. Thus  $f\in L^2(0,T, L^2_a(\omega) )$. Let 
\[
u(x,t):=e^{Kt} v(x) \hat u ( y(x) , t).
\]
Then $u$ solves \eqref{B100}-\eqref{B400} and $u(.,T)=0$. 
 \qed

\section*{Acknowledgments} The authors wish to thank Enrique Zuazua who suggested to consider $L^1$
coefficients instead of $L^\infty$ coefficients.  
\bibliographystyle{abbrv}        
\bibliography{PEREF}           

\begin{thebibliography}{10}

\bibitem{ACF}
F.~Alabau-Boussouira, P.~Cannarsa, and G.~Fragnelli.
\newblock Carleman estimates for degenerate parabolic operators with
  applications to null controllability.
\newblock {\em J. Evol. Equ.}, 6(2):161--204, 2006.

\bibitem{AE}
G.~Alessandrini and L.~Escauriaza.
\newblock Null-controllability of one-dimensional parabolic equations.
\newblock {\em ESAIM Control Optim. Calc. Var.}, 14(2):284--293, 2008.

\bibitem{BCG}
K.~Beauchard, P.~Cannarsa, and R.~Guglielmi.
\newblock Null controllability of grushin-type operators in dimension two.
\newblock {\em J. Eur. Math. Soc.}, 16:67--101, 2014.

\bibitem{BDL}
A.~Benabdallah, Y.~Dermenjian, and J.~Le~Rousseau.
\newblock Carleman estimates for the one-dimensional heat equation with a
  discontinuous coefficient and applications to controllability and an inverse
  problem.
\newblock {\em J. Math. Anal. Appl.}, 336(2):865--887, 2007.

\bibitem{BR}
G.~Birkhoff and G.-C. Rota.
\newblock {\em Ordinary Differential Equations}.
\newblock Ginn-Blaisdell, Boston, 1962.

\bibitem{bogachev}
V.~I. Bogachev.
\newblock {\em Measure theory. {V}ol. {I}, {II}}.
\newblock Springer-Verlag, Berlin, 2007.

\bibitem{BL}
U.~Boscain and C.~Laurent.
\newblock The laplace-beltrami operator in almost-riemannian geometry.
\newblock {\em Ann. Inst. Fourier (Grenoble)}, to appear.

\bibitem{CFR}
P.~Cannarsa, G.~Fragnelli, and D.~Rocchetti.
\newblock Controllability results for a class of one-dimensional degenerate
  parabolic problems in nondivergence form.
\newblock {\em J. Evol. Equ.}, 8(4):583--616, 2008.

\bibitem{CMV04}
P.~Cannarsa, P.~Martinez, and J.~Vancostenoble.
\newblock Persistent regional null controllability for a class of degenerate
  parabolic equations.
\newblock {\em Commun. Pure Appl. Anal.}, 3(4):607--635, 2004.

\bibitem{CMV08}
P.~Cannarsa, P.~Martinez, and J.~Vancostenoble.
\newblock Carleman estimates for a class of degenerate parabolic operators.
\newblock {\em SIAM J. Control Optim.}, 47(1):1--19, 2008.

\bibitem{CMV09}
P.~Cannarsa, P.~Martinez, and J.~Vancostenoble.
\newblock Carleman estimates and null controllability for boundary-degenerate
  parabolic operators.
\newblock {\em C. R. Math. Acad. Sci. Paris}, 347(3-4):147--152, 2009.

\bibitem{cazacu}
C.~Cazacu.
\newblock Controllability of the heat equation with an inverse-square potential
  localized on the boundary.
\newblock {\em SIAM J. Control Optim.}, to appear.

\bibitem{colton}
D.~Colton.
\newblock Integral operators and reflection principles for parabolic equations
  in one space variable.
\newblock {\em J. Differential Equations}, 15:551--559, 1974.

\bibitem{ervedoza}
S.~Ervedoza.
\newblock Control and stabilization properties for a singular heat equation
  with an inverse-square potential.
\newblock {\em Comm. Partial Differential Equations}, 33(10-12):1996--2019,
  2008.

\bibitem{FR}
H.~Fattorini and D.~Russell.
\newblock Exact controllability theorems for linear parabolic equations in one
  space dimension.
\newblock {\em Arch. Rational Mech. Anal.}, 43(4):272--292, 1971.

\bibitem{FZ}
E.~Fern{\'a}ndez-Cara and E.~Zuazua.
\newblock On the null controllability of the one-dimensional heat equation with
  {BV} coefficients.
\newblock {\em Comput. Appl. Math.}, 21(1):167--190, 2002.
\newblock Special issue in memory of Jacques-Louis Lions.

\bibitem{FT}
C.~Flores and L.~de~Teresa.
\newblock Carleman estimates for degenerate parabolic equations with first
  order terms and applications.
\newblock {\em C. R. Math. Acad. Sci. Paris}, 348(7-8):391--396, 2010.

\bibitem{FI}
A.~V. Fursikov and O.~Y. Imanuvilov.
\newblock {\em Controllability of evolution equations}, volume~34 of {\em
  Lecture Notes Series}.
\newblock Seoul National University Research Institute of Mathematics Global
  Analysis Research Center, 1996.

\bibitem{gordon}
R.~A. Gordon.
\newblock {\em The integrals of {L}ebesgue, {D}enjoy, {P}erron, and
  {H}enstock}, volume~4 of {\em Graduate Studies in Mathematics}.
\newblock American Mathematical Society, Providence, RI, 1994.

\bibitem{hajlasz93}
P.~Haj{\l}asz.
\newblock Change of variables formula under minimal assumptions.
\newblock {\em Colloq. Math.}, 64(1):93--101, 1993.

\bibitem{harris}
B.~J. Harris.
\newblock Asymptotics of eigenvalues for regular sturm-liouville problems.
\newblock {\em J. Math. Anal. Appl.}, 183:25--36, 1994.

\bibitem{imanuvilov}
O.~Y. Imanuvilov.
\newblock Controllability of parabolic equations.
\newblock {\em Mat. Sb.}, 186(6):109--132, 1995.

\bibitem{jones}
B.~Jones~Jr.
\newblock A fundamental solution for the heat equation which is supported in a
  strip.
\newblock {\em J. Math. Anal. Appl.}, 60(2):314--324, 1977.

\bibitem{komornik}
V.~Komornik.
\newblock {\em Exact controllability and stabilization}.
\newblock RAM: Research in Applied Mathematics. Masson, Paris; John Wiley \&
  Sons, Ltd., Chichester, 1994.
\newblock The multiplier method.

\bibitem{Laroc2000PhD}
B.~Laroche.
\newblock {\em Extension de la notion de platitude \`a des syst\`emes d\'ecrits
  par des \'equations aux d\'eriv\'ees partielles lin\'eaires.}
\newblock PhD thesis, Ecole des Mines de Paris, 2000.

\bibitem{LarocM2000MTNS}
B.~Laroche and P.~Martin.
\newblock Motion planning for a 1-{D} diffusion equation using a
  {B}runovsky-like decomposition.
\newblock In {\em Proceedings of the International Symposium on the
  Mathematical Theory of Networks and Systems (MTNS)}, 2000.

\bibitem{LMR}
B.~Laroche, P.~Martin, and P.~Rouchon.
\newblock Motion planning for the heat equation.
\newblock {\em Int J Robust Nonlinear Control}, 10(8):629--643, 2000.

\bibitem{lerousseau}
J.~Le~Rousseau.
\newblock Carleman estimates and controllability results for the
  one-dimensional heat equation with {BV} coefficients.
\newblock {\em J. Differential Equations}, 233(2):417--447, 2007.

\bibitem{LR}
G.~Lebeau and L.~Robbiano.
\newblock Contr\^ole exact de l'\'equation de la chaleur.
\newblock {\em Comm. Partial Differential Equations}, 20(1-2):335--356, 1995.

\bibitem{LK}
W.~A.~J. Luxemburg and J.~Korevaar.
\newblock Entire functions and {M}\"untz-{S}z\'asz type approximation.
\newblock {\em Trans. Amer. Math. Soc.}, 157:23--37, 1971.

\bibitem{MRR1D}
P.~Martin, L.~Rosier, and P.~Rouchon.
\newblock Null controllability of the 1{D} heat equation using flatness.
\newblock In {\em 1st IFAC workshop on Control of Systems Governed by Partial
  Differential Equations (CPDE2013)}, pages 7--12, 2013.

\bibitem{MRR2D}
P.~Martin, L.~Rosier, and P.~Rouchon.
\newblock Null controllability of the 2{D} heat equation using flatness.
\newblock In {\em 52nd IEEE Conference on Decision and Control}, pages
  3738--3743, 2013.

\bibitem{MRRND}
P.~Martin, L.~Rosier, and P.~Rouchon.
\newblock Null controllability of the heat equation using flatness.
\newblock {\em Automatica J. IFAC}, 50(12):3067--3076, 2014.

\bibitem{MRRPE}
P.~Martin, L.~Rosier, and P.~Rouchon.
\newblock Null controllability using flatness: a case study of a 1-d heat
  equation with discontinuous coefficients.
\newblock In {\em 14th European Control Conference (ECC15)}, 2015.

\bibitem{Meurer2012book}
T.~Meurer.
\newblock {\em Control of Higher--Dimensional PDEs: Flatness and Backstepping
  Designs}.
\newblock Communications and Control Engineering. Springer, 2012.

\bibitem{rudin}
W.~Rudin.
\newblock {\em Real and complex analysis}.
\newblock McGraw-Hill Book Co., third edition, 1987.

\bibitem{spa}
S.~Sp{\u{a}}taru.
\newblock An absolutely continuous function whose inverse function is not
  absolutely continuous.
\newblock {\em Note Mat.}, 23(1):47--49, 2004/05.

\bibitem{VZ}
J.~Vancostenoble and E.~Zuazua.
\newblock Null controllability for the heat equation with singular
  inverse-square potentials.
\newblock {\em J. Funct. Anal.}, 254(7):1864--1902, 2008.

\end{thebibliography}

\end{document}